\documentclass[11pt,oneside,a4paper]{article}

\usepackage{amsmath}
\usepackage{amssymb}
\usepackage{amsthm}
\usepackage{algorithmic}

\usepackage{authblk}
\usepackage{cite}
\usepackage[english]{babel}
\usepackage[T1]{fontenc}
\usepackage{caption}
\usepackage{graphicx}
\usepackage{mathtools}
\usepackage{empheq}
\usepackage{subfig}
\usepackage[utf8]{inputenc}
\usepackage{pgf}
\usepackage{tikz}
\usepackage{url}
\usepackage{sectsty}
\usepackage{fourier}
\usepackage{times}

\usepackage{hyperref}
\hypersetup{colorlinks,
            breaklinks,
            urlcolor=blue,
            linkcolor=blue,
            citecolor=blue}

\usepackage[noabbrev,capitalise,nameinlink]{cleveref}
\allsectionsfont{\sffamily}

\graphicspath{ {./Figures/} }
\newtheorem{theorem}{Theorem}

\newtheorem{proposition}{Proposition}
\theoremstyle{definition}
\newtheorem{definition}{Definition} 
\theoremstyle{definition}
\newtheorem{algorithm}{Algorithm} 
\theoremstyle{definition}

\numberwithin{equation}{section}

\newcommand{\grad}{\mathrm{grad}}

\usepackage[activate={true,nocompatibility},
            final,tracking=false,
            kerning=true,
            spacing=true,
            factor=1100,
            stretch=10,
            shrink=10]{microtype}
\microtypecontext{spacing=nonfrench}
\DeclareMathAlphabet{\mathcal}{OMS}{cmsy}{m}{n}

\title{Dissipative numerical schemes on Riemannian manifolds with applications to gradient flows}
\author{Elena Celledoni$^*$, Sølve Eidnes$^*$, Brynjulf Owren$^*$, Torbjørn Ringholm%
  \thanks{Department of Mathematical Sciences, Norwegian University of Science and Technology, N–7491 Trondheim\\
  This work was supported by the European Union's Horizon 2020 research and innovation programme under the Marie Skłodowska-Curie grant agreement No. 691070.}}
\begin{document}
\maketitle

\begin{abstract}
This paper concerns an extension of discrete gradient methods to finite-dimensional Riemannian manifolds termed discrete Riemannian gradients, and their application to dissipative ordinary differential equations. This includes Riemannian gradient flow systems which occur naturally in optimization problems. The Itoh--Abe discrete gradient is formulated and applied to gradient systems, yielding a derivative-free optimization algorithm. The algorithm is tested on two eigenvalue problems and two problems from manifold valued imaging: InSAR denoising and DTI denoising.

\textbf{Keywords}: Geometric integration, discrete gradients, Riemannian manifolds, numerical optimization, InSAR denoising, DTI denoising. 

\textbf{Classification}: 49M37, 53B99, 65K10, 92C55, 90C26, 90C30, 90C56
\end{abstract}

\section{Introduction}

When designing and applying numerical schemes for solving systems of ODEs and PDEs there are several important properties which serve to distinguish schemes, one of which is the preservation of geometric features of the original system. The field of geometric integration encompasses many types of numerical schemes for ODEs and PDEs specifically designed to preserve one or more such geometric features; a non-exhaustive list of features includes symmetry, symplecticity, first integrals (or energy), orthogonality, and manifold structures such as Lie group structure \cite{hairer2006geometric}.  Energy conserving methods have a successful history in the field of numerical integration of ODEs and PDEs. In a similar vein, numerical schemes with guaranteed dissipation are useful for solving dissipative equations such as gradient systems.

As seen in \cite{humphries1994runge}, any Runge--Kutta method can be dissipative when applied to gradient systems as long as step sizes are chosen small enough; less severe but still restrictive conditions for dissipation in Runge--Kutta methods are presented in \cite{hairer2013energy}. 
In \cite{gonzalez96tia}, Gonzalez introduces the notion of discrete gradient schemes with energy preserving properties, later expanded upon to include dissipative systems in \cite{mclachlan99giu}. 
These articles consider ODEs in Euclidian spaces only with the exception of \cite{hairer2013energy} where the authors also consider Runge--Kutta methods on manifolds defined by constraints.
Unlike the Runge--Kutta methods, discrete gradient methods are dissipative for all step sizes, meaning one can employ adaptive time steps while retaining convergence toward fixed points \cite{sato2015lyapunov}.
However, one may experience a practical step size restriction when applying discrete gradient methods to very stiff problems, due to the lack of $L$-stability as seen when applying the Gonzalez and mean value discrete gradients to problems with quadratic potentials \cite{hairer2013energy}\cite{hairer1996solving}.
Motivated by their work on Lie group methods, the energy conserving discrete gradient method was generalized to ODEs on manifolds, and Lie groups particularly, in \cite{Celledoni2014977} where the authors introduce the concept of discrete differentials. 
In \cite{celledoni2018energy}, this concept is specialized in the setting of Riemannian manifolds. 
To the best of our knowledge, the discrete gradient methods have not yet been formulated for dissipative ODEs on manifolds. Doing so is the central purpose of this article. 

One of the main reasons for generalizing discrete gradient methods to dissipative systems on manifolds is that gradient systems are dissipative, and gradient flows are natural tools for optimization problems which arise in e.g. manifold-valued image processing and eigenvalue problems. The goal is then to find one or more stationary points of the gradient flow of a functional $V: M \rightarrow \mathbb{R}$, which correspond to critical points of $V$. This approach is, among other optimization methods, presented in \cite{absil2009optimization}.  Since gradient systems occur naturally on Riemannian manifolds, it is natural to develop our schemes in a Riemannian manifold setting. 

A similarity between the optimization algorithms in \cite{absil2009optimization} and the manifold valued discrete gradient methods in \cite{Celledoni2014977} is their use of retraction mappings. Retraction mappings were introduced for numerical methods in \cite{shub1986some}, see also \cite{adler2002newton}; they are intended as computationally efficient alternatives to parallel transport on manifolds. Our methods will be formulated as a framework using general discrete gradients on general Riemannian manifolds with general retractions. We will consider a number of specific examples that illustrate how to apply the procedure in practical problems. 

As detailed in \cite{Grimm_2016} and \cite{ringholm2017variational}, using the Itoh--Abe discrete gradient \cite{Itoh1988}, one can obtain an optimization scheme for $n$-dimensional problems with a limited degree of implicitness. At every iteration, one needs to solve $n$ decoupled scalar nonlinear subequations, amounting to $\mathcal{O}(n)$ operations per step. In other discrete gradient schemes a system of $n$ coupled nonlinear equations must be solved per iteration, amounting to $\mathcal{O}(n^2)$ operations per step. The Itoh--Abe discrete gradient method therefore appears to be well suited to large-scale problems such as image analysis problems, and so it seems natural to apply our new methods to image analysis problems on manifolds, see \cref{subsect:mfldvalimg}. In \cite{Celledoni2014977}, the authors generalize the average vector field \cite{harten83oud} and midpoint \cite{gonzalez96tia} discrete gradients, but not the Itoh--Abe discrete gradient, to Lie groups and homogeneous manifolds. A novelty of this article is the formulation of the Itoh--Abe discrete gradient for problems on manifolds.

As examples we will consider two eigenvalue finding problems, in addition to the more involved problems of denoising InSAR and DTI images using total variation (TV) regularization \cite{weinmann2014total}. The latter two problems we consider as real applications of the algorithm. The two eigenvalue problems are included mostly for the exposition and illustration of our methods, as well as for testing convergence properties.

The paper is organized as follows: Below, we introduce notation and fix some fundamental definitions used later on. In the next section, we formulate the dissipative problems we wish to solve. In section 3, we present the discrete Riemannian gradient (DRG) methods, a convergence proof for the family of optimization methods obtained by applying DRG methods to Riemannian gradient flow problems, the Itoh--Abe discrete gradient generalized to manifolds, and the optimization algorithm obtained by applying the Itoh--Abe DRG to the gradient flow problem. In section 4, we provide numerical experiments to illustrate the use of DRGs in optimization, and in the final section we present conclusions and avenues for future work. 

\subsection*{Notation and preliminaries}
Some notation and definitions used in the following are summarized below. For a more thorough introduction to the concepts, see e.g. \cite{lang2012fundamentals} or \cite{lee2006riemannian}.

\begin{table}[ht]
{
\caption{Notational conventions}
\begin{center}
\begin{tabular}{c c}
Notation & Description  \\
\hline
$M$      & $n$-dimensional Riemannian manifold \\
$T_pM $   & tangent space at $p \in M$ with zero vector $0_p$ \\
$T^*_pM $   & cotangent space at $p \in M$ \\
$TM $   & tangent bundle of $M$ \\
$T^*M $   & cotangent bundle of $M$ \\
$\mathfrak{X}(M)$ & space of vector fields on $M$\\
$g(\cdot,\cdot)$ & Riemannian metric on $M$ \\
$\| \cdot \|_p$ & Norm induced on $T_pM$ by $g$\\
$\{E_l\}_{l = 1}^n$ & $g$-orthogonal basis of $T_pM$
\end{tabular}
\end{center}
}
\end{table}

On any differentiable manifold there is a duality pairing $\left\langle \cdot,\cdot \right\rangle: T^*M \times TM \rightarrow \mathbb{R} $ which we will denote as $\left\langle \omega,v \right\rangle = \omega(v)$. Furthermore, the Riemannian metric sets up an isomorphism between $TM$ and $T^*M$ via the linear map $v \mapsto g(v,\cdot)$. This map and its inverse, termed the musical isomorphisms, are known as the flat map $^\flat: TM \rightarrow T^*M$ and sharp map $^\sharp: T^*M \rightarrow TM$, respectively. The applications of these maps are also termed index raising and lowering when considering the tensorial representation of the Riemannian metric. Note that with the above notation we have the idiom $x^{\flat}(y) = \left\langle x^{\flat},y \right\rangle = g(x,y)$. 

On a Riemannian manifold, one can define gradients: For $V \in C^\infty(M)$, the (Riemannian) gradient with respect to $g$,  $\grad_g V \in \mathfrak{X}(M)$, is the unique vector field such that $g(\grad_g V, X) = \left\langle \mathrm{d}V,X \right\rangle$ for all $X \in \mathfrak{X}(M)$. In the language of musical isomorphisms, $\grad_g V = (\mathrm{d}V)^{\sharp}$. For the remainder of this article, we will write $\grad V$ for the gradient and assume that it is clear from the context which $g$ is to be used.
 
Furthermore, the \textit{geodesic} between $p$ and $q$ is the unique curve of minimal length between $p$ and $q$, providing a distance function $d_M: M \times M \rightarrow \mathbb{R}$. The geodesic $\gamma$ passing through $p$ with tangent $v$ is given by the Riemannian exponential at $p$, $\gamma(t) = \exp_p(tv)$. For any $p$, $\exp_p$ is a diffeomorphism on a neighbourhood $N_p$ of $0_p$, The image $\exp_p(S_p)$ of any star-shaped subset $S_p \subset N_p$ is called a normal neighbourhood of $p$, and on this, $\exp_p$ is a radial isometry, i.e. $d_M(p,\exp_p(v)) = \| v \|_p$ for all $v \in S_p$.

\section{The problem}

We will consider ordinary differential equations (ODEs) of the form
\begin{align}
\dot{u} = F(u), \quad u(0) = u^0 \in M,
\label{eq:ODE}
\end{align}
where $F \in \mathfrak{X}(M)$ has an associated energy $V: M \rightarrow \mathbb{R}$ dissipating along solutions of \eqref{eq:ODE}. That is, with $u(t)$ a solution of \eqref{eq:ODE}:
\begin{align*}
\dfrac{\mathrm{d}}{\mathrm{d}t} V(u) = \left\langle \mathrm{d}V(u),\dot{u} \right\rangle = \left\langle \mathrm{d}V(u),F(u) \right\rangle = g(\mathrm{grad}V(u),F(u)) \leq 0.
\end{align*}
An example of such an ODE is the gradient flow. Given an energy $V$, the gradient flow of $V$ with respect to a Riemannian metric $g$ is 
\begin{align}
\dot{u} = -\grad V(u), \label{eq:gradflow}
\end{align}
which is dissipative since if $u(t)$ solves \eqref{eq:gradflow}, we have
\begin{align*}
\dfrac{\mathrm{d}}{\mathrm{d}t} V(u) = -g\left(\grad V(u),\grad V(u)\right) \leq 0 .
\end{align*}
%From \cite{jost2008riemannian}, we know that if $V$ is $C^2$ with bounded first and second derivatives, then solutions $u(t)$ of (\ref{eq:gradflow}) exist for all initial values $u^0$. Under three additional conditions we can guarantee that $\lim_{t \rightarrow \infty} u(t)$ exists and is a critical point of $V$, and that $u(t)$ converges exponentially to its limit as $t \rightarrow \infty$. These conditions are:
%\begin{enumerate}
%  \item $V(u(t))$ is bounded.
%  \item $V$ is a Morse function. That is, all critical points $p^*$ of $V$ are nondegenerate in the sense that $d^2V(p^*)$ has nonzero eigenvalues only.
%  \item The Palais-Smale condition holds for $V$. That is, every sequence $\{ x_k \}_{k \in \mathbb{N}}$ for which $|V(x_k)|$ is bounded and $\lim_{k \rightarrow \infty} \| dV(x_k) \|= 0$, contains a convergent subsequence.
%\end{enumerate}
%Assumptions 2 and 3 above also guarantee that $V$ has finitely many critical points in any bounded region of $M$.
%\newline
%\newline
\textbf{Remark:} This setting can be generalized by an approach similar to\cite{mclachlan99giu}. Suppose there exists a (0,2) tensor field $h$ on $M$ such that $h(x,x) \leq 0$. We can associate to $h$ the (1,1) tensor field $H: TM \rightarrow TM$ given by $Hx = h(x,\cdot)^{\sharp}$. Consider the system
\begin{align}
\dot{u} = H\mathrm{grad}V(u). \label{eq:gradform}
\end{align}
This system dissipates $V$, since
\begin{align*}
\dfrac{\mathrm{d}}{\mathrm{d}t}V(u) &= \left\langle \mathrm{d}V(u),  H\mathrm{grad}V(u)\right\rangle\\
&= g\left(\mathrm{grad}V(u),  H\mathrm{grad}V(u)\right)\\
&= h \left( \mathrm{grad}V(u), \mathrm{grad}V(u)\right) \leq 0.
\end{align*}
Any dissipative system of the form \eqref{eq:ODE} can be written in this form on the set $M\backslash \{p \in M: g(F(p),\grad V(p)) = 0 \}$ since, given $F$ and $V$, we can construct $h$ as follows: 
\begin{align*}
h= \dfrac{1}{g(F,\grad V)}F^{\flat} \otimes F^{\flat}. 
\end{align*}
If $F = -\grad V$, we take $h = -g$ such that $H$ becomes $-\mathrm{Id}$, and recover \eqref{eq:gradflow}. In the following, we mainly discuss the case $F = -\grad V$ for the sake of notational clarity.

\section{Numerical scheme}
The discrete differentials in \cite{Celledoni2014977} are formulated such that they may be used on non-Riemannian manifolds. Since we restrict ourselves to Riemannian manifolds, we define their analogues: discrete Riemannian gradients. As with the discrete differentials, we shall make use of retractions as defined in \cite{shub1986some}.
\begin{definition} Let $\phi: TM \rightarrow M$ and denote by $\phi_p$ the restriction of $\phi$ to $T_pM$. Then, $\phi$ is a \textit{retraction} if the following conditions are satisfied:
\begin{itemize}
\item $\phi_p$ is smooth and defined in an open ball $B_{r_p(0_p)}$ of radius $r_p$ around $0_p$, the zero vector in $T_pM$.
\item $\phi_p(v) = p$ if and only if $v = 0_p$.
\item Identifying $T_{0_p} T_p M \simeq T_pM$, $\phi_p$ satisfies 
\begin{align*}
d\phi_p\big|_{0_p} = \mathrm{id}_{T_p M},
\end{align*} 
where $\mathrm{id}_{T_p M}$ denotes the identity mapping on $T_p M$.
\end{itemize}
\end{definition}
From the inverse function theorem it follows that for any $p$, there exists a neighbourhood $U_{p,\phi} \in T_p M$ of $0_p$, such that $\phi_p : U_{p,\phi} \rightarrow \phi_p(U_{p,\phi}) $ is a diffeomorphism. In general, $\phi_p$ is not a diffeomorphism on the entirety of $T_pM$ and so all the following schemes must be considered local in nature. The canonical retraction on a Riemannian manifold is the Riemannian exponential. It may be computationally expensive to evaluate even if closed expressions for geodesics are known, and so one often wishes to come up with less costly retractions if possible. We are now ready to introduce the notion of discrete Riemannian gradients.

\begin{definition}
Given a retraction $\phi$, a function  $c:M \times M \rightarrow M$ where $c(p,p) = p $ for all $ p \in M$ and a continuous $V:M \rightarrow \mathbb{R}$, then $\overline{\mathrm{grad}}V:M\times M \rightarrow TM$ is a discrete Riemannian gradient of $V$ if it is continuous and, for all $p,q \in U_{c(p,q),\phi}$,
\begin{align}
V(q) - V(p) &= g \left( \overline{\mathrm{grad}}V(p,q),\phi_{c(p,q)}^{-1}(q) - \phi_{c(p,q)}^{-1}(p)\right) \label{eq:DGprop1}\\
\overline{\mathrm{grad}}V(p,p) &=  \mathrm{grad}V|_{p}. \label{eq:DGprop2}
\end{align} 
\end{definition}
We formulate a numerical scheme for equation \eqref{eq:gradflow} based on this definition. Given times $0 = t_0 < t_1 < ... $, let $u^k$ denote the approximation to $u(t_k)$ and let $\tau_k = t_{k+1} - t_{k}$. Then, we take
\begin{align}
u^{k+1} &= \phi_{c^k}\left(W(u^k,u^{k+1})\right) \label{eq:scheme1}\\
  W(u^k,u^{k+1}) &= \phi_{c^k}^{-1}(u^k) - \tau_k  \, \overline{\mathrm{grad}}V(u^k,u^{k+1})  \label{eq:scheme2}
\end{align}
where  $c^k = c(u^k, u^{k+1})$ and 
In the above and all of the following, we assume that $u^k$ and $u^{k+1}$ lie in $U_{c^k,\phi} \cap S_{c^k}$. The following proposition verifies that the scheme is dissipative.
\begin{proposition} The sequence $\{ u^k \}_{k \in \mathbb{N}}$ generated by the DRG scheme \eqref{eq:scheme1}-\eqref{eq:scheme2} satisfies $V(u^{k+1}) - V(u^k) \leq 0$ for all $k \in \mathbb{N}$. \label{prop:dissip}
\end{proposition}
\begin{proof}
Using property \eqref{eq:DGprop1} and equations \eqref{eq:scheme1} and \eqref{eq:scheme2}, we get 
\begin{align*}
V(u^{k+1}) - V(u^k) &= g\left(\overline{\mathrm{grad}}V(u^k,u^{k+1}),\phi_{c^k}^{-1}(u^{k+1}) - \phi_{c^k}^{-1}(u^k)\right)\\
&= g \left(\overline{\mathrm{grad}}V(u^k,u^{k+1}),W(u^k,u^{k+1}) - \phi_{c^k}^{-1}(u^k)\right)\\
&= -\tau_k g\left(\overline{\mathrm{grad}}V(u^k,u^{k+1}), \overline{\mathrm{grad}}V(u^k,u^{k+1}) \right) \leq 0
\end{align*}
\end{proof}
\textbf{Remark:} This extends naturally to schemes for \eqref{eq:gradform} by exchanging \eqref{eq:scheme2} for 
\begin{align*}
  W(u^k,u^{k+1}) &= \phi_{c^k}^{-1}(u^k) + \tau_k  \, \overline{H}_{(u^k, u^{k+1})} \, \overline{\mathrm{grad}}V(u^k,u^{k+1}), 
\end{align*}
where $\overline{H}_{(p,q)}$ is the (1,1) tensor associated with a negative semi-definite (0,2) tensor field $\overline{h}_{(p,q)} : T_{c(p,q)}M \times T_{c(p,q)}M \rightarrow \mathbb{R}$ approximating $h|_p$ consistently.

Two DRGs, the AVF DRG and the Gonzalez DRG, can be easily found by index raising the discrete differentials defined in \cite{Celledoni2014977}. We will later generalize the Itoh--Abe discrete gradient, but first we present a proof that the DRG scheme converges to a stationary point when used as an optimization algorithm. We will need the following definition of coercivity:
\begin{definition}
A function $V: M\rightarrow \mathbb{R}$ is \textit{coercive} if, for all $v \in M$, every sequence $\{u^k \}_{k \in \mathbb{N}} \subset M $ such that $\lim \limits_{k \rightarrow \infty} d_M(u^k,v) = \infty$ also satisfies $\lim \limits_{k \rightarrow \infty} V(u^k) = \infty$.
\end{definition}
We will also need the following theorem from \cite{udriste1994convex}, concerning the boundedness of the sublevel sets $M_\mu = \{u \in M : V(u) \leq \mu \}$ of $V$:
\begin{theorem}
Assume $M$ is unbounded. Then the sublevel sets of $V:M\rightarrow \mathbb{R}$ are bounded if and only if $V$ is coercive. \label{theo:coerciveBdd}
\end{theorem}
\begin{proof}
 See \cite{udriste1994convex}, Theorem 8.6, Chapter 1 and the remarks below it.  
\end{proof}
Equipped with this, we present the following theorem, the proof of which is inspired by that of the convergence theorem in \cite{Grimm_2016}.

\begin{theorem}\label{theo:convergence}
Assume that $M$ is geodesically complete, that $V: M \rightarrow R$ is coercive, bounded from below and continuously differentiable, and that $\overline{\mathrm{grad}}V$ is continuous. Then, the iterates $\{u^k\}_{k\in \mathbb{N}}$ produced by applying the discrete Riemannian gradient scheme \eqref{eq:scheme1}-\eqref{eq:scheme2} with time steps $0 < \tau_{min} \leq \tau_k \leq \tau_{max}$ and $c^k = u^k$ or $c^k = u^{k+1}$, to the gradient flow of $V$ satisfy
\begin{align*}
\lim \limits_{k \rightarrow \infty} \overline{\mathrm{grad}}V(u^k,u^{k+1}) =  \lim \limits_{k \rightarrow \infty} \mathrm{grad} V(u^k) = 0.
\end{align*}
Additionally, there exists at least one accumulation point $u^*$ of $\{u^k\}_{k\in \mathbb{N}}$, and any such accumulation point satisfies $\mathrm{grad}V(u^*) = 0$. 
\end{theorem}
\begin{proof}
Since $V$ is bounded from below and by \cref{prop:dissip}, we have
\begin{align*}
C \leq V(u^{k+1}) \leq V(u^k) \leq ... \leq V(u^0)
\end{align*}
such that, by the monotone convergence theorem, $V^* := \lim_{k \rightarrow \infty} V(u^k)$ exists. Furthermore, by property \eqref{eq:DGprop1} and using the scheme \eqref{eq:scheme1}-\eqref{eq:scheme2}:
\begin{align*}
\dfrac{1}{\tau_k} \left\|\phi_{c^k}^{-1}(u^k) - \phi_{c^k}^{-1}(u^{k+1}) \right\|_{c^{k}}^2 &= \tau_k \left\|\overline{\mathrm{grad}}V(u^k,u^{k+1}) \right\|_{c^k}^2\\
&= g\left(\overline{\mathrm{grad}} V (u^k,u^{k+1}), \phi_{c^k}^{-1}(u^k) - \phi_{c^k}^{-1}(u^{k+1}) \right)\\
&= V(u^k) - V(u^{k+1}).
\end{align*}
From this, it is clear that for any $i,j \in \mathbb{N}$,
\begin{align*}
\sum \limits_{k = i}^{j-1} \tau_k \left\|\overline{\mathrm{grad}}V(u^k,u^{k+1}) \right\|_{c^k}^2 =  V(u^i) - V(u^{j}) \leq V(u^0) - V^*
\end{align*}
and 
\begin{align*}
\sum \limits_{k = i}^{j-1} \dfrac{1}{\tau_k} \left \|\phi_{c^k}^{-1}(u^k) - \phi_{c^k}^{-1}(u^{k+1}) \right\|_{c^k}^2 =  V(u^i) - V(u^j) \leq V(u^0) - V^*.
\end{align*}
In particular,
\begin{align*}
\sum \limits_{k = 0}^{\infty} \left\|\overline{\mathrm{grad}}V(u^k,u^{k+1}) \right\|_{c^k}^2 \leq \dfrac{V(u^0) - V^*}{\tau_{min}},
\end{align*}
and 
\begin{align*}
\sum \limits_{k = 0}^{\infty} \left \|\phi_{c^k}^{-1}(u^k) - \phi_{c^k}^{-1}(u^{k+1}) \right \|_{c^k}^2 \leq \tau_{max} \left(V(u^0) - V^*\right),
\end{align*}
meaning
\begin{align*}
\lim_{k \rightarrow \infty}\left \|\overline{\mathrm{grad}}V(u^k,u^{k+1}) \right\|_{c^k} &= 0,\\
\lim_{k \rightarrow \infty}\left \|\phi_{c^k}^{-1}(u^k) - \phi_{c^k}^{-1}(u^{k+1}) \right\|_{c^k} &= 0.
\end{align*}
Since $u^{k+1}$ is in a normal neighbourhood of $c^k$,
\begin{align}
d_M(c^k,u^{k+1}) &= d_M(c^k,\exp_{c^k}(\exp_{c^k}^{-1}(u^{k+1}))) = \|\exp^{-1}_{c^k}(u^{k+1}) \|_{c^k}. \label{eq:firstineq}
\end{align}
Introduce $\psi_{c^k} : T_{c^k}M \rightarrow T_{c^k}M $ by $\psi_{c^k} = \exp^{-1}_{c^k} \circ \, \phi_{c^k} $. Since both $\exp$ and $\phi$ are retractions,
\begin{align*}
\psi_{c^k}(0_{c^k}) &= 0_{c^k},\\
D\psi_{c^k}|_{0_{c^k}} &= \mathrm{id}_{T_{c^k} M}.
\end{align*}
Thus, per definition of Fr\'{e}chet derivatives,
\begin{align*}
\psi_{c^k}(x) - \psi_{c^k}(0_{c^k}) - D\psi_{c^k}|_{0_{c^k}} x = \psi_{c^k}(x) - x = o(x),
\end{align*}
in particular: choosing $x = \phi^{-1}_{c^k}(u^{k+1})$ we get
\begin{align*}
\exp^{-1}_{c^k}(u^{k+1}) - \phi^{-1}_{c^k}(u^{k+1})= o(\|\phi^{-1}_{c^k}(u^{k+1})\|_{c^k}),
\end{align*}
meaning 
\begin{align}
\|\exp^{-1}_{c^k}(u^{k+1})\|_{c^k} \leq \| \phi^{-1}_{c^k}(u^{k+1}) \|_{c^k} +  o(\|\phi^{-1}_{c^k}(u^{k+1})\|_{c^k}) . \label{eq:secondineq}
\end{align}
Taking $c^k = u^k$ and combining \eqref{eq:firstineq} and \eqref{eq:secondineq} we find
\begin{align*}
d(u^k,u^{k+1}) = \|\exp^{-1}_{c^k}(u^{k+1}) \|_{c^k} \leq \| \phi^{-1}_{c^k}(u^{k+1}) \|_{c^k} + o(\|\phi^{-1}_{c^k}(u^{k+1})\|_{c^k}).
\end{align*}
Hence, since $\left \|\phi_{c^k}^{-1}(u^k) - \phi_{c^k}^{-1}(u^{k+1}) \right\|_{c^k} = \left \| \phi_{c^k}^{-1}(u^{k+1}) \right\|_{c^k}$ when $c^k = u^k$,
\begin{align}
\lim_{k \rightarrow \infty} d(u^k,u^{k+1}) \leq \lim_{k \rightarrow \infty} \left \|\phi_{c^k}^{-1}(u^k) - \phi_{c^k}^{-1}(u^{k+1}) \right\|_{c^k} = 0. \label{eq:dconv}
\end{align}
Note that we can exchange the roles of $u^k$ and $u^{k+1}$ and obtain the same result.

Since $V$ is bounded from below, the sublevel sets $M_\mu$ of $V$ are the preimages of the closed subsets $[C,\mu]$ and are hence closed as well. Since $V$ is assumed to be coercive, by \cref{theo:coerciveBdd} the $M_\mu$ are bounded, and so since $M$ is geodesically complete, by the Hopf-Rinow theorem the $M_\mu$ are compact \cite{udriste1994convex}.  In particular, $M_{V(u^0)}$ is compact such that $\overline{\mathrm{grad} }V$ is uniformly continuous on $M_{V(u^0)} \times M_{V(u^0)}$ by the Heine-Cantor theorem. This means that for any $\epsilon > 0$ there exists $\delta > 0$ such that if $d_{M\times M}((u^k,u^{k+1}),(u^{k},u^k)) = d_{M}(u^k,u^{k+1}) < \delta$, then 
\begin{align*}
\left  \|\overline{\mathrm{grad}} V (u^k,u^{k+1}) - \mathrm{grad} V(u^k) \right \|_{c^k}  = \left \|\overline{\mathrm{grad}} V (u^k,u^{k+1}) - \overline{\mathrm{grad}} V(u^k,u^k) \right \|_{c^k}  < \epsilon.
\end{align*}
Since $d_{M}(u^k,u^{k+1}) \rightarrow 0$, given $\epsilon > 0$ there exists $K$ such that for all $k > K$, 
\begin{align*}
\left \|  \mathrm{grad} V(u^k)\right \|_{c^k}  \leq \left \|\overline{\mathrm{grad}} V (u^k,u^{k+1}) - \mathrm{grad} V(u^k) \right \|_{c^k}  +\left \|\overline{\mathrm{grad}} V (u^k,u^{k+1}) \right\|_{c^k}  \leq  2\epsilon.
\end{align*}
This means 
\begin{align*}
\lim_{k \rightarrow \infty}  \mathrm{grad} V(u^k)  = 0.
\end{align*}
Since $M_{V(u^0)}$ is compact, there exists a convergent subsequence $\{u^{k_l} \} $ with limit $u^*$. Since $V$ is continuously differentiable,
\begin{align*}
\textrm{grad}V(u^*) = \lim_{l \rightarrow \infty} \textrm{grad}V(u^{k_l}) = 0.
\end{align*} 
\end{proof}
\textbf{Remark:} In the above proof, we assumed $c^k = u^k$ or $c^k = u^{k+1}$. Although these choices may be desirable for practical purposes, as discussed in the next subsection, one can also make a more general choice. Specifically, if $\phi = \exp$ and $c^k$, let $\gamma^k(t)$ be the geodesic between $u^k$ and $u^{k+1}$ such that
\begin{align*}
\gamma^k(t) = \exp_{u^k}(tv^k) 
\end{align*}
where $v^k = \exp^{-1}_{u^k}(u^{k+1})$. Then, taking $c^k = \gamma^k(s)$ for some $s \in [0,1]$, uniqueness of geodesics implies that
\begin{align*}
\exp_{c^k}(t\dot{\gamma}^k(s)) = \exp_{u^k}((t+s)v^k).
\end{align*}
Hence,
\begin{align*}
\exp^{-1}_{c^k}(u^k) = -s\dot{\gamma}^k(s), \qquad \exp^{-1}_{c^k}(u^{k+1}) = (1-s)\dot{\gamma}^k(s), 
\end{align*}
and so, since geodesics are constant speed curves: 
\begin{align*}
d(u^k,u^{k+1}) = \|v\|_{u^k} = \| \dot{\gamma}^k(s) \|_{c^k} =  \|   \exp^{-1}_{c^k}(u^k) - \exp^{-1}_{c^k}(u^{k+1}) \|_{c^k}.
\end{align*}
This means that \eqref{eq:dconv} holds in this case. No other arguments in \cref{theo:convergence} are affected.
\subsection{Itoh--Abe discrete Riemannian gradient}
The Itoh--Abe discrete gradient \cite{Itoh1988} can be generalized to Riemannian manifolds.
\begin{proposition}
Given a continuously differentiable energy $V:M \rightarrow \mathbb{R}$ and an orthogonal basis $\{E_j\}_{j = 1}^n$ for $T_{c(u,v)}M$ such that
\begin{align*}
\phi_c^{-1}(v) - \phi_c^{-1}(u) = \sum \limits_{i=1}^n \alpha_i E_i,
\end{align*}
define $\overline{\mathrm{grad}}_{\mathrm{IA}}V : M \times M \rightarrow T_{c(u,v)}M $ by
\begin{align*}
\overline{\mathrm{grad}}_{\mathrm{IA}} V(u,v) = \sum \limits_{j=1}^{n} a_j E_j,
\end{align*}
where
\begin{align*}
a_j &= 
\begin{cases}
\dfrac{V(w_j) - V(w_{j-1})}{\alpha_j}, \quad &\alpha_j \neq 0\\
g(\mathrm{grad}V(w_{j-1}),d \phi_c \big|_{\eta_{j-1}} E_j), \quad &\alpha_j = 0.
\end{cases}\\
w_j &= \phi_c(\eta_j), \quad \eta_j = \phi_c^{-1}(u) + \sum \limits_{i=1}^j \alpha_i E_i.
\end{align*}
Then, $\overline{\mathrm{grad}}_{\mathrm{IA}} V$ is a discrete Riemannian gradient.
\end{proposition}
\begin{proof}
Continuity of $\overline{\mathrm{grad}}_{\mathrm{IA}} V$ can be seen from the smoothness of the local coordinate frame $\{E_j\}_{j = 1}^n$ and from the continuity of the $a_j(\alpha_j)$:
\begin{align*}
\lim_{\alpha_j \rightarrow 0} a_j(\alpha_j) &= \lim_{\alpha_j \rightarrow 0} \dfrac{V\left(\phi_c \left(\eta_{j-1} + \alpha_j E_j\right)\right) - V\left(\phi_c \left(\eta_{j-1} \right)\right) }{\alpha_j}\\
&= \dfrac{\mathrm{d}}{\mathrm{d}\alpha_j}\bigg|_{\alpha_j = 0} V\left(\phi_c \left(\eta_{j-1} + \alpha_j E_j\right)\right) \\
&= \left\langle \mathrm{d}V\left(\phi_c \left(\eta_{j-1} \right)\right), d\phi_c \big|_{\eta_{j-1}}E_j   \right\rangle\\
&= g(\mathrm{grad}V(w_{j-1}),d \phi_c\big|_{\eta_{j-1}} E_j).
\end{align*}
Property \eqref{eq:DGprop1} holds since 
\begin{align*}
g \left( \overline{\mathrm{grad}}_{\mathrm{IA}} V(u,v),\phi_c^{-1}(v) - \phi_c^{-1}(u) \right) &= \sum \limits_{i = 1}^{n} \sum_{j = 1}^{n} \alpha_i a_j g(E_i ,E_j)\\
&= \sum_{j = 1}^{n} V(w_j) - V(w_{j-1})\\
&= V(w_n) - V(w_0)\\
&= V(v) - V(u).
\end{align*}
Furthermore, \eqref{eq:DGprop2} holds since when $v = u$, all $\alpha_j = 0$ and $c(u,v) = u$  so that
\begin{align*}
\overline{\mathrm{grad}}_{\mathrm{IA}} V(u,u) = \sum_{j= 1}^n g(\mathrm{grad} V(u),E_j) E_j = \mathrm{grad} V(u). 
\end{align*}
\end{proof}
The map $\overline{\mathrm{grad}}_{\mathrm{IA}} V$ is called the Itoh--Abe discrete Riemannian gradient. For the Itoh--Abe DRG to be a computationally viable option it is important to compute the $\alpha_i$ efficiently. Consider for instance the gradient flow system. Applying the Itoh--Abe DRG to this we get the scheme
\begin{align*}
u^{k+1} &= \phi_{c^k} \left(W(u^k,u^{k+1})\right),\\
W(u^k,u^{k+1}) &= \phi_{c^k}^{-1}(u^k) - \tau_k \overline{\mathrm{grad}}_{\mathrm{IA}}  V(u^k,u^{k+1}),
\end{align*}
meaning 
\begin{align*}
\phi_{c^k}^{-1}(u^{k+1}) - \phi_{c^k}^{-1}(u^k) = - \tau_k \overline{\mathrm{grad}}_{\mathrm{IA}} V(u^k,u^{k+1}),
\end{align*}
and in coordinates
\begin{align*}
\sum \limits_{i=1}^n \alpha_i E_i = - \tau_k \sum \limits_{j=1}^{n} \dfrac{V(w_j) - V(w_{j-1})}{\alpha_j} E_j,
\end{align*}
so that the $\alpha_i$ are found by solving the $n$ coupled equations
\begin{align*}
\alpha_i = - \tau_k \dfrac{V(w_i) - V(w_{i-1})}{\alpha_i}.
\end{align*}
Note that these equations in general are fully implicit in the sense that they require knowledge of the endpoint $u^{k+1}$ since the $w_i$ are dependent on $c^k$. However, if we take $c^k = u^k$, there is no dependency on the endpoint and all the above equations become scalar, although one must solve them successively. For this choice of $c^k$ we present, as \cref{alg:DG}, a procedure for solving the gradient flow problem on a Riemannian manifold with Riemannian metric $g$ using the Itoh--Abe DRG.

\begin{algorithm}[DRG-OPTIM]
\begin{algorithmic}
\label{alg:DG}
\STATE{}
\STATE{Choose $tol > 0$ and $ u^0 \in M.$ Set $ k = 0. $}
\REPEAT 
\STATE{Choose $\tau_k $  and  an orthogonal basis $\{E_i^k\}_{i=1}^n$  for $T_{u^k}M$}
\STATE{$v_0^k = u^k$}
\STATE{$w_{0}^k = \phi_{u^k}^{-1}(v_{0}^k)$}
\FOR{$j = 1,...,n$}
\STATE{Solve $\alpha_j^k  = - \tau_k \left(V\left(\phi_{u^k}(w_{j-1}^{k} + \alpha_j^k E_j^k)\right) - V\left(v_{j-1}^{k}\right)\right)/\alpha_j^k$}
\STATE{$w_{j}^k = w_{j-1}^k + \alpha_j^k E_j^k $}
\STATE{$v_j^k = \phi_{u^k}(w_j^k) $}
\ENDFOR
\STATE{$u^{k+1} = v_{n}^k$}
\STATE{$k = k + 1$}
\UNTIL{$\left(V(u^k) - V(u^{k-1})\right)/V(u^0) < tol$}
\end{algorithmic}
\end{algorithm}

There is a caveat to this algorithm in that the $\alpha_j^k$ should be easy to compute. 
For example, it is important that the $E_j$ and $\phi$ are chosen such that the difference $V(\phi_{u^k}(w_{j-1}^{k} + \alpha_j^k E_j^k)) - V(v_{j-1}^{k})$ is cheap to evaluate. 
In many cases, $M$ has a natural interpretation as a submanifold of Euclidean space defined locally by constraints $g:\mathbb{R}^m \rightarrow \mathbb{R}^n$, $M = \{ y \in U \subset \mathbb{R}^m : g(y) = 0 \}$.
Then, one may find $\{E_j\}_{j = 1}^n$ as an orthogonal basis for $\mathrm{ker} \,g'(c)$ and define $\phi_c$ implicitly by taking $q = \phi_c(v)$ such that $q - (c+v) \in (T_cM)^{\perp}$ and $g(q) = 0$, as detailed in \cite{celledoni2002class}.
This requires the solution of a nonlinear system of equations for every coordinate update, which is computationally demanding compared to evaluating explicit expressions for $\{E_j\}_{j = 1}^n$ and $\phi_c$ as is possible in special cases, such as those considered in \cref{sect:numexp}.
To compute the $\alpha_j^k$ at each coordinate step one can use any suitable root finder, yet to stay in line with the derivative-free nature of \cref{alg:DG}, one may wish to use a solver like the Brent--Dekker algorithm \cite{brent1971algorithm}. Also worth noting is that the parallelization procedure used in \cite{ringholm2017variational} works for \cref{alg:DG} as well. 

\section{Numerical experiments} \label{sect:numexp}
This section concerns four applications of DRG methods to gradient flow systems. In each case, we specify all details needed to implement \cref{alg:DG} the manifold $M$, retraction $\phi$, and basis vectors $\{ E_k\}$. The first two examples are eigenvalue problems, included to illuminate implementational issues with examples in a familiar setting. We do not claim that our algorithm is competitive with other eigenvalue solvers, but include these examples for the sake of exposition and to have problems with readily available reference solutions. The first of these is a simple Rayleigh quotient minimization problem, where issues of computational efficiency are raised. The second one concerns the Brockett flow on $\mathrm{SO}(m)$, the space of orthogonal $m \times m$ matrices with unit determinant, and serves as an example of optimization on a Lie group. The remaining two problems are examples of manifold-valued image analysis problems concerning Interferometric Synthetic Aperture Radar (InSAR) imaging and Diffusion Tensor Imaging (DTI), respectively. 
Specifically, the problems concern total variation denoising of images obtained through these techniques \cite{weinmann2014total}. 
The experiments do not consider the quality of the solution paths, i.e. numerical accuracy. For experiments of this kind, we refer to \cite{celledoni2018energy}.

All programs used in the following were implemented as MATLAB functions, with critical functions implemented in C using the MATLAB EXecutable (MEX) interface when necessary. The code was executed using MATLAB (2017a release) running on a Mid 2014 MacBook Pro with a four-core 2.5 GHz Intel Core i7 processor and 16 GB of 1600 MHz DDR3 RAM. We used a C language port of the built-in MATLAB function $\mathtt{fzero}$ for the Brent-Dekker algorithm implementation.

\subsection{Eigenvalue problems}
As an expository example, our first problem consists of finding the smallest eigenvalue/vector pair of a symmetric $m \times m$ matrix $A$ by minimizing its Rayleigh quotient. We shall solve this problem using both the extrinsic and intrinsic view of the $(m-1)$-sphere. In the second example we consider the different approach to the eigenvalue problem proposed by Brockett in \cite{brockett1988dynamical}. Here, the gradient flow on $\mathrm{SO}(m)$ produces a diagonalizing matrix for a given symmetric matrix.

\subsubsection{Eigenvalues via Rayleigh quotient minimization}
In our first example, we wish to compute the smallest eigenvalue of a symmetric matrix $A \in \mathbb{R}^{m \times m}$ by minimizing the Rayleigh quotient
\begin{align*}
V(u) = u^TA u
\end{align*}
with $u$ on the $(m-1)$-sphere $S^{m-1}$.

Taking the extrinsic view, we regard $S^{m-1}$ as a submanifold in $\mathbb{R}^m$, equipped with the standard Euclidian metric $g(x,y) = x^T y$. In this representation, $T_u S^{m-1}$ is the hyperplane tangent to $u$, i.e.  $T_u S^{m-1} = \{x \in \mathbb{R}^m : x^T u = 0 \}.
$ A natural choice of retraction is
\begin{align*}
\phi_p(x) = \dfrac{p + x}{\|p+x\|}.
\end{align*}
There is a difficulty with this $\phi$; it does not preserve sparsity, meaning \cref{alg:DG} will be inefficient as discussed above. To see this, consider that at each time step, to find the $\alpha_j^k$, we must compute the difference
\begin{align*}
V(z_j^{k}) - V(z_{j-1}^{k}) = (z_j^{k})^TAz_j^{k}- (z_{j-1}^{k})^TAz_{j-1}^{k}
\end{align*}
for some $z_{j-1}^{k}, z_{j}^{k}\in S^{m-1}$. We can compute this efficiently if $z_{j}^{k}= z_{j-1}^{k} + \delta$, where $\delta$ is sparse. Then, 
\begin{align*}
V(z_{j}^{k}) - V(z_{j-1}^{k}) = 2(z_{j-1}^{k})^TA\delta + \delta^T A\delta,
\end{align*}
which is efficient since one may assume $Az_{j-1}^{k}$ to be precomputed so that the computational cost is limited by the sparsity of $\delta$. In our case, we have 
\begin{align*}
z_{j-1}^{k} = \phi_c(w_{j-1}^{k} ), \qquad z_{j}^{k} = \phi_c(w_{j-1}^{k} + \alpha_j^k E_j).
\end{align*}
However, with $\phi_c$ as above, $ \delta = \phi_c(w_{j-1}^{k} + \alpha_j^k E_j) - \phi_c(w_{j-1}^{k} )$ is non-sparse, and so computing the energy difference is costly. 

Next, let us consider the intrinsic view of $S^{m-1}$, representing it in spherical coordinates $\theta \in \mathbb{R}^{m-1}$ by
\begin{align*}
u_1(\theta) &= \cos(\theta_1),\\
u_r(\theta) &= \cos(\theta_r)\prod \limits_{i=1}^{r-1} \sin (\theta_i), \quad 1 < r < m,\\
u_m(\theta) &= \prod \limits_{i=1}^{m-1} \sin (\theta_i).
\end{align*}
Due to the simple structure of $\mathbb{R}^{m-1}$, we take $\phi_\theta(\eta) = \theta + \eta$. Then, we have
\begin{align*}
u_r(\phi_\theta(\alpha E_l)) = u_r(\theta + \alpha E_l) =
\begin{cases}
u_r(\theta), &r < l\\
\dfrac{\cos(\theta_l + \alpha)}{\cos(\theta_l)}u_r(\theta) , &r = l \\
\dfrac{\sin(\theta_l + \alpha)}{\sin(\theta_l)}u_r(\theta) , &r > l.
\end{cases}
\end{align*}
Using this relation, the energy difference after a coordinate update becomes:
\begin{align*}
V(u(\theta + \alpha E_l)) - V(u(\theta)) &= 2\kappa_{1l}\sum_{i = 1}^{l-1}  u_i(\theta)u_l(\theta) A_{il} + 2\kappa_{2l}\sum_{i = 1}^{l-1}\sum_{j = l+1}^{m} u_i(\theta)u_j(\theta) A_{ij} \\
&+ 2\kappa_{3l}\sum_{j = l+1}^{m} u_l(\theta)u_j(\theta) A_{lj} + \kappa_{4l} \sum_{i = l+1}^{m} \! \sum_{j = l+1}^{m} u_i(\theta)u_j(\theta) A_{ij}.\\
&+\kappa_{5l} u_l(\theta)u_l(\theta) A_{ll},
\end{align*}
with
\begin{align*}
\kappa_{1l} = c_l - 1,  \quad  
\kappa_{2l} = s_l- 1,  \quad 
\kappa_{3l} = s_l c_l - 1, \quad 
\kappa_{4l} = s_l^2 - 1,  \quad 
\kappa_{5l} =  c_l^2  - 1,
\end{align*}
where
\begin{align*}
c_l = \dfrac{\cos(\theta_l + \alpha)}{\cos(\theta_l)}, \quad  s_l = \dfrac{\sin(\theta_l + \alpha)}{\sin(\theta_l) }.
\end{align*}
With prior knowledge of $V(u(\theta))$ (and thus the four partial sums in the difference), evaluating $V(u(\theta + \alpha E_l)) - V(u(\theta))$ amounts to five scalar multiplications and four scalar additions after evaluating the $\kappa_i^l$. With correct bookkeeping, new sums can be evaluated from previous sums after coordinate updates, reducing the computational complexity of the algorithm. Although not producing an algorithm competitive with standard eigenvalue solvers, this example demonstrates that the correct choice of coordinates is vital to reducing the computational complexity of the Itoh--Abe DRG method.

\subsubsection{Eigenvalues via Brockett flow}
\label{subsect:Brockett}
Among other things, the article of Brockett \cite{brockett1988dynamical} discusses how one may find the eigenvalues of a symmetric matrix $A$ by solving the following gradient flow problem on $M = \mathrm{SO}(m)$:
\begin{align}
\dot{Q} = -Q(D Q^T A Q - Q^T A Q D) \label{eq:Brockett}
\end{align}
Here, $D$ is a real diagonal matrix with non-repeated entries. It can be shown that $\lim_{t \rightarrow \infty} Q = Q^*$, where $(Q^*)^T A Q^* = \Lambda$ is diagonal and hence contains the eigenvalues of $A$, ordered as the entries of $D$. Equation \eqref{eq:Brockett} is the gradient flow of the energy
\begin{align}
V(Q) = \mathrm{tr}(AQ^T D Q) \label{eq:BrockettDoubleBracket}
\end{align}
with respect to the trace metric on $\mathrm{SO}(m)$. One can check that $\mathrm{SO}(m)$ is a Lie group \cite{warner2013foundations}, with Lie algebra
\begin{align*}
\mathrm{\mathfrak{so}}(m) = \{B \in\mathbb{R}^{m \times m} : B^T = -B \}.
\end{align*}
Also, since $\mathrm{SO}(m)$ is a matrix Lie group, the exponential coincides with the matrix exponential. However, we may consider using some other function as a retraction, such as the Cayley transform $\phi: \mathfrak{so}(m) \rightarrow \mathrm{SO}(m)$ given by
\begin{align*}
\phi(B) = (I-B)^{-1}(I+B).
\end{align*}

\begin{figure}[ht]
                \centering
        \begin{minipage}{.5\textwidth}
            \centering
            \includegraphics[width=0.98\linewidth]{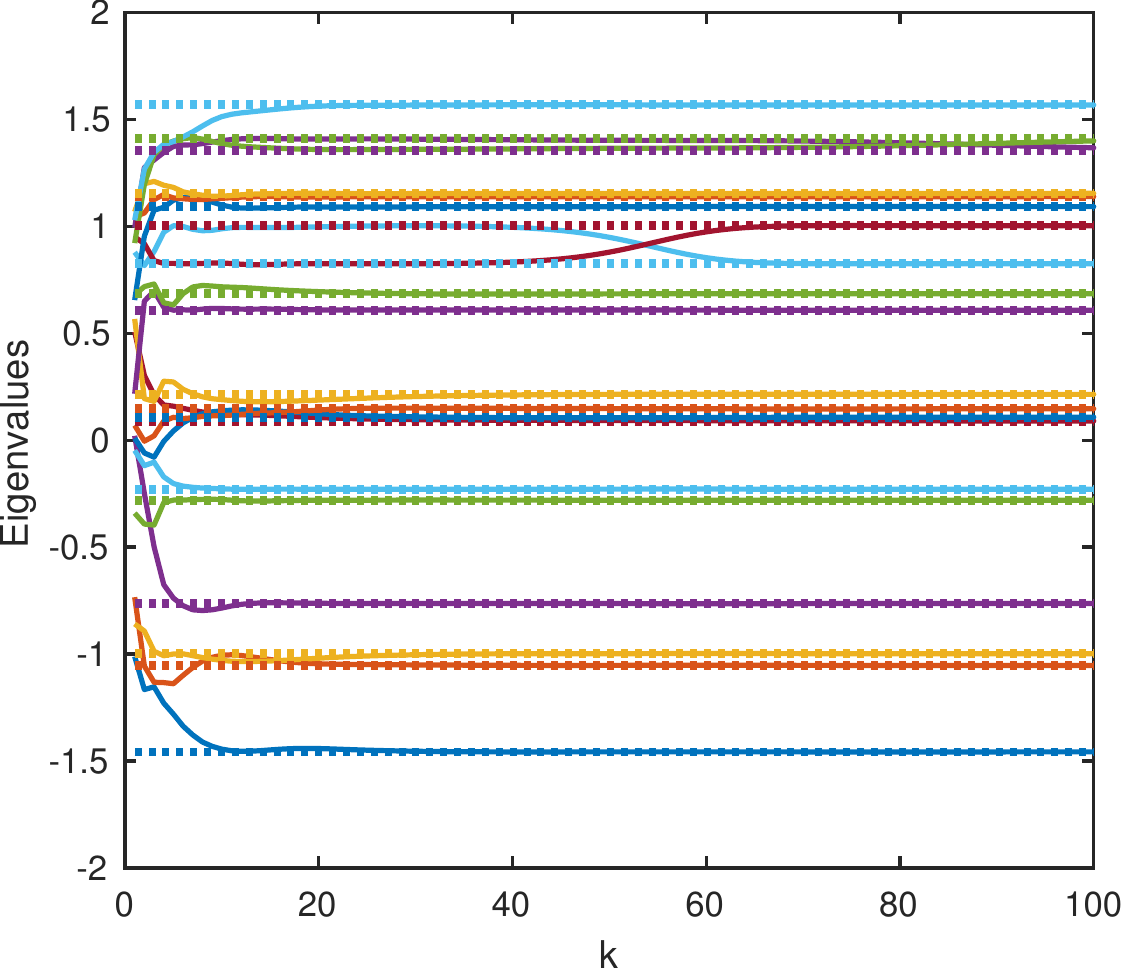}
        \end{minipage}%
        \begin{minipage}{.5\textwidth}
            \centering
            \includegraphics[width=0.98\linewidth]{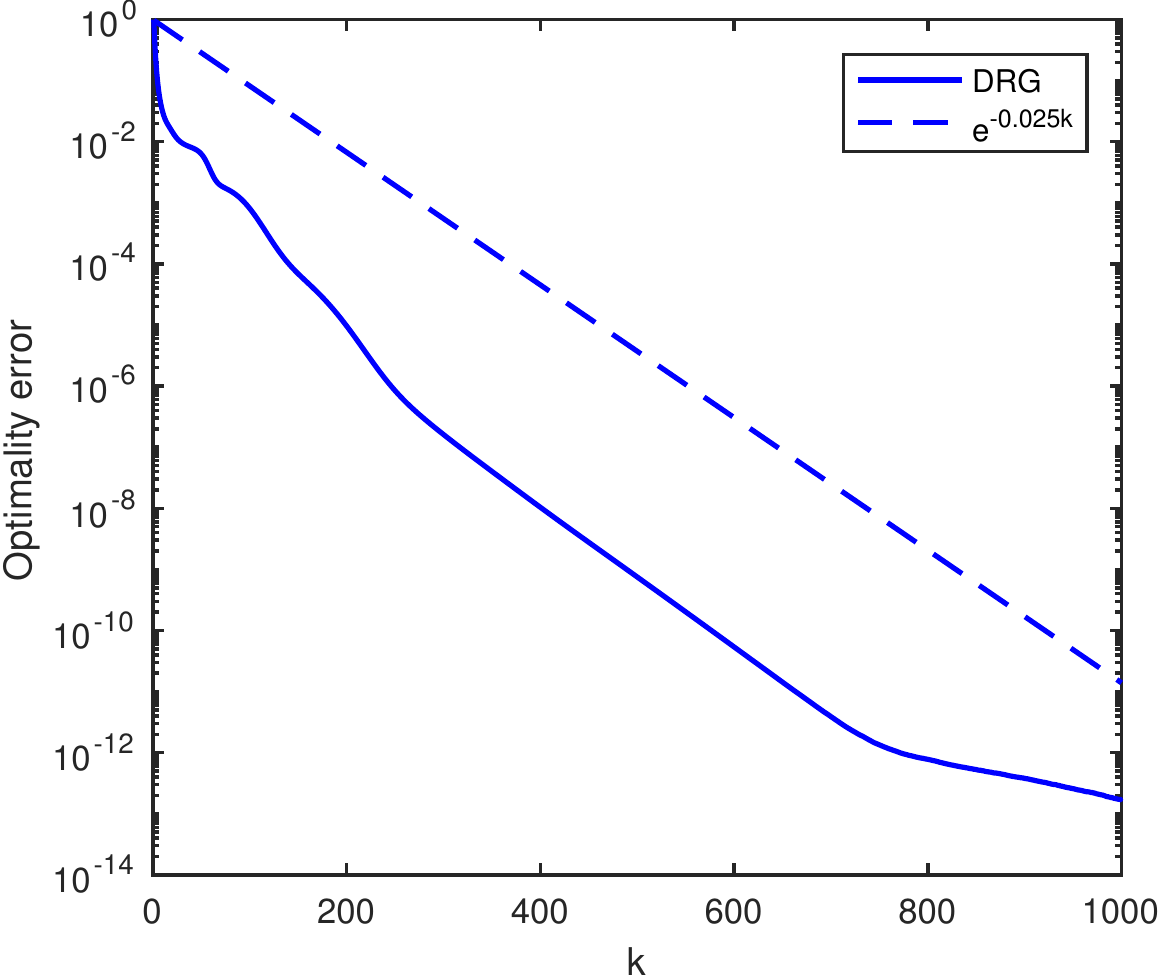}
        \end{minipage}%
        \caption{Brockett flow with $\tau_k = 0.1$ and 20 eigenvalues. Random initial matrix. Left: Evolution of eigenvalues. Right: Optimality error $(V(u^k) - V^*)/(V(u^0) - V^*)$.}
        \label{fig:Brockett}
\end{figure}
\cref{fig:Brockett} shows the results of numerical tests with constant time step $\tau_k = 0.1$ and $m = 20$. In the left hand panel, the evolution of the diagonal values of $Q^{k}A Q^{k}$ compared to the spectrum of $A$ is shown; it is apparent that the diagonal values converge to the eigenvalues. The right hand panel shows the convergence rate of \cref{alg:DG} to the minimal value $V^*$ as computed with eigenvalues and eigenvectors from MATLAB's $\mathtt{eigen}$ function. It would appear that the convergence rate is linear, meaning $\|D - (Q^{k+1})^TAQ^{k+1} \| = C \|D - (Q^{k})^TAQ^{k} \|$, with $C < 1$, which corresponds to an exponential reduction in $\|D - (Q^{k})^TAQ^{k} \|$. No noteworthy difference was observed when using the matrix exponential in place of the Cayley transform.

\subsection{Manifold valued imaging} \label{subsect:mfldvalimg}
In the following two examples we will consider problems from  manifold valued 2D imaging. We will in both cases work on a product manifold $\mathcal{M} = M^{l\times m}$ consisting of $l\times m$ copies of an underlying data manifold $M$. An element of $M$ will in this case be called an \textit{atom}, as opposed to the regular term \textit{pixel}. As explained in \cite{lee2006riemannian}, product manifolds of Riemannian manifolds are again Riemannian manifolds. The tangent spaces of product manifolds have a natural structure as direct sums, with $T_{(u_{11}, u_{12}, ..., u_{lm} )}\mathcal{M} = \bigoplus_{i,j=1}^{l,m} T_{u_{ij}}M$, which induces a natural Riemannian metric $\mathcal{G}: T\mathcal{M} \times T\mathcal{M} \rightarrow \mathbb{R}$ fiberwise as
\begin{align*}
\mathcal{G}_{(u_{11}, u_{12}, ..., u_{lm} )}((x_{11},...,x_{lm}),(y_{11},...,y_{lm})) = \sum \limits_{i,j = 1}^{l,m} g_{u_{ij}}(x_{ij},y_{ij}).
\end{align*}
Also, given a retraction $\phi : TM \rightarrow M$, one can define a retraction $\Phi: T\mathcal{M} \rightarrow \mathcal{M}$ fiberwise as 
\begin{align*}
\Phi_{(u_{11}, u_{12}, ..., u_{lm} )}(x_{11},...,x_{lm}) = (\phi_{u_{11}}(x_{11}), \phi_{u_{12}}(x_{12}), ..., \phi_{u_{lm}}(x_{lm})).
\end{align*} 
Discrete gradients were first used in optimization algorithms for image analysis in \cite{Grimm_2016} and \cite{ringholm2017variational}. As an example of a manifold-valued imaging problem, consider Total Variation (TV) denoising of manifold valued images \cite{weinmann2014total}, where one wishes to minimize, based on generalizations of the $L^\beta$ and $L^{\gamma}$ norms:
\begin{align}
V(u) = \dfrac{1}{\beta}\sum_{i,j = 1}^{l,m}\mathrm{d}(u_{ij},s_{ij})^\beta + \lambda \left( \sum_{i,j = 1}^{l-1,m}\mathrm{d}(u_{ij},u_{i+1,j})^\gamma  + \sum_{i,j = 1}^{l,m-1}\mathrm{d}(u_{ij},u_{i,j+1})^\gamma \right).
\label{eq:TV_mfld}
\end{align}
Here, $s = (s_{11},...,s_{lm}) \in \mathcal{M}$ is the input image, $u = (u_{11},...,u_{lm}) \in \mathcal{M}$ is the output image, $\lambda$ is a regularization strength constant, and $\mathrm{d}$ is a metric on $M$, which we will take to be the geodesic distance induced by $g$.

\subsubsection{InSAR image denoising}
We first consider Interferometric Synthetic Aperture Radar (InSAR) imaging, used in earth observation and terrain modelling \cite{rosen2000synthetic}. In InSAR imaging, terrain elevation is measured by means of phase differences between laser pulses reflected from a surface at different times. Thus, the atoms  $g_{ij}$ are elements of $M = S^1$, represented by their phase angles: $-\pi < g_{ij} \leq \pi$. After processing, the phase data is \textit{unwrapped} to form a single, continuous image of displacement data \cite{goldstein1988satellite}. The natural distance function in this representation is the angular distance
\begin{align*}
\mathrm{d}(\varphi,\theta) = 
\begin{cases}
|\varphi-\theta|, \quad &|\varphi-\theta| \leq \pi\\
2 \pi - |\varphi-\theta|, \quad &|\varphi-\theta| > \pi.
\end{cases}
\end{align*}
Also, $T_\varphi M$ is simply $\mathbb{R}$, and $\phi$ is given, with $\underset{ 2 \pi}+ $ denoting addition modulo $2\pi$, as:
\begin{align*}
\phi_\varphi(\theta_\varphi) = (\theta \underset{ 2 \pi}+ (\varphi+\pi)) - \pi.
\end{align*}

\begin{figure}[ht!]
                \centering
        \begin{minipage}{.45\textwidth}
            \centering
            \includegraphics[width=0.98\linewidth]{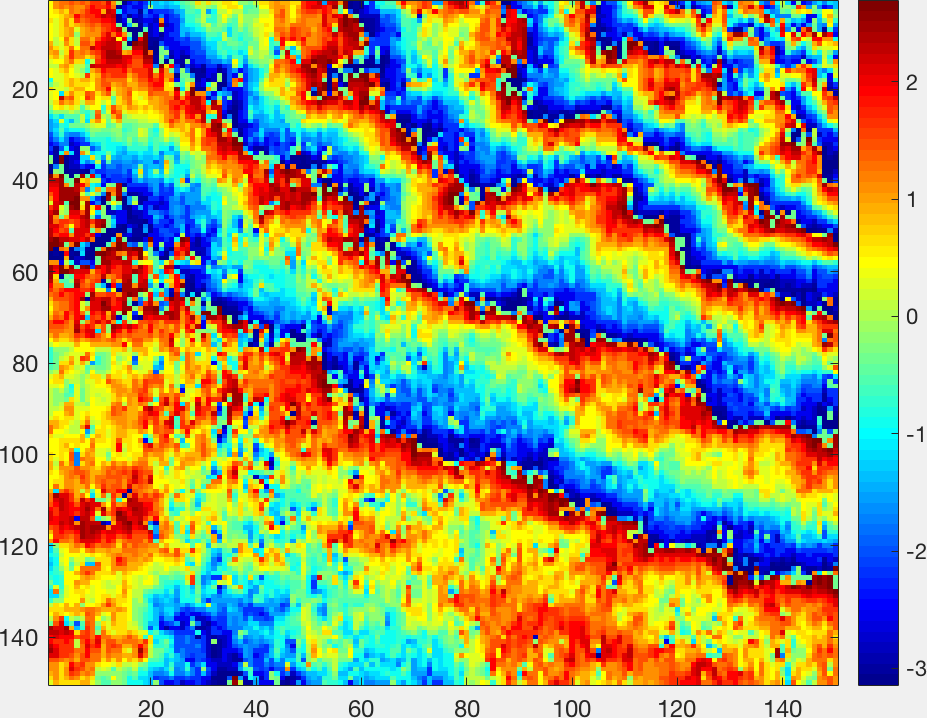}
        \end{minipage}%
        \begin{minipage}{.45\textwidth}
            \centering
            \includegraphics[width=0.98\linewidth]{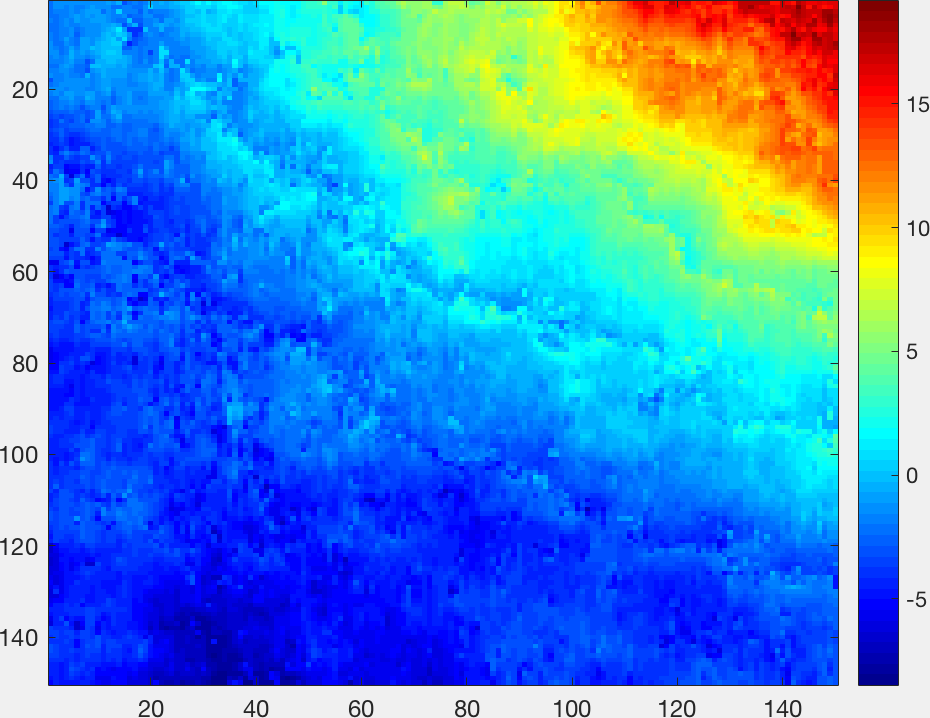}
        \end{minipage}
		\\
%		\vspace{3pt}
%                \centering
%        \begin{minipage}{.5\textwidth}
%            \centering
%            \includegraphics[width=0.98\linewidth]{InSAR_L1.png}
%        \end{minipage}%
%        \begin{minipage}{.5\textwidth}
%            \centering
%            \includegraphics[width=.98\linewidth]{InSAR_L1_uwr.png}
%        \end{minipage}
%        \\
		\vspace{3pt}
                \centering
        \begin{minipage}{.45\textwidth}
            \centering
            \includegraphics[width=0.98\linewidth]{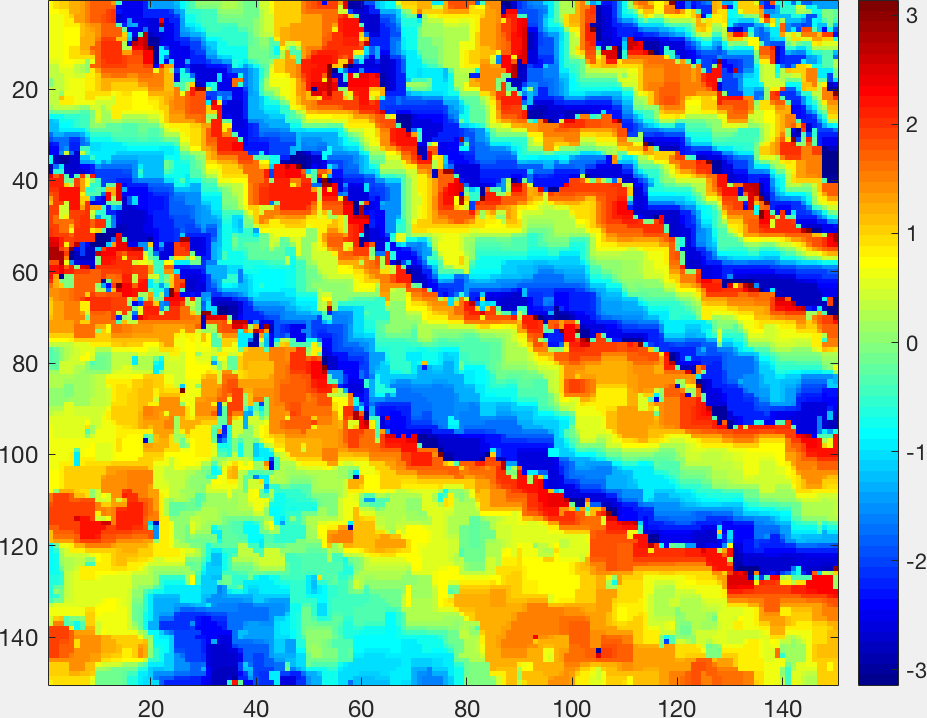}
        \end{minipage}%
        \begin{minipage}{.45\textwidth}
            \centering
            \includegraphics[width=.98\linewidth]{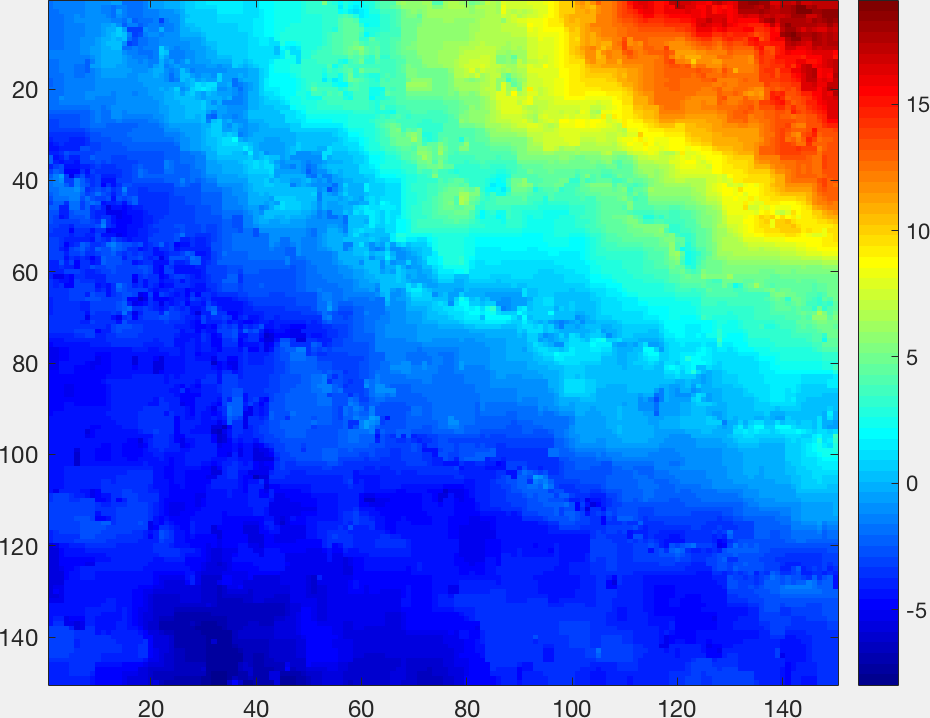}
        \end{minipage}
        \caption{Left column: Interferogram. Right column: Phase unwrapped image. Top row: Original image. Bottom row: L2 fidelity denoising, $\lambda = 0.3$.}
        \label{fig:InSAR}
\end{figure}

\begin{figure}[ht!]
            \centering
            \includegraphics[width=0.5\linewidth]{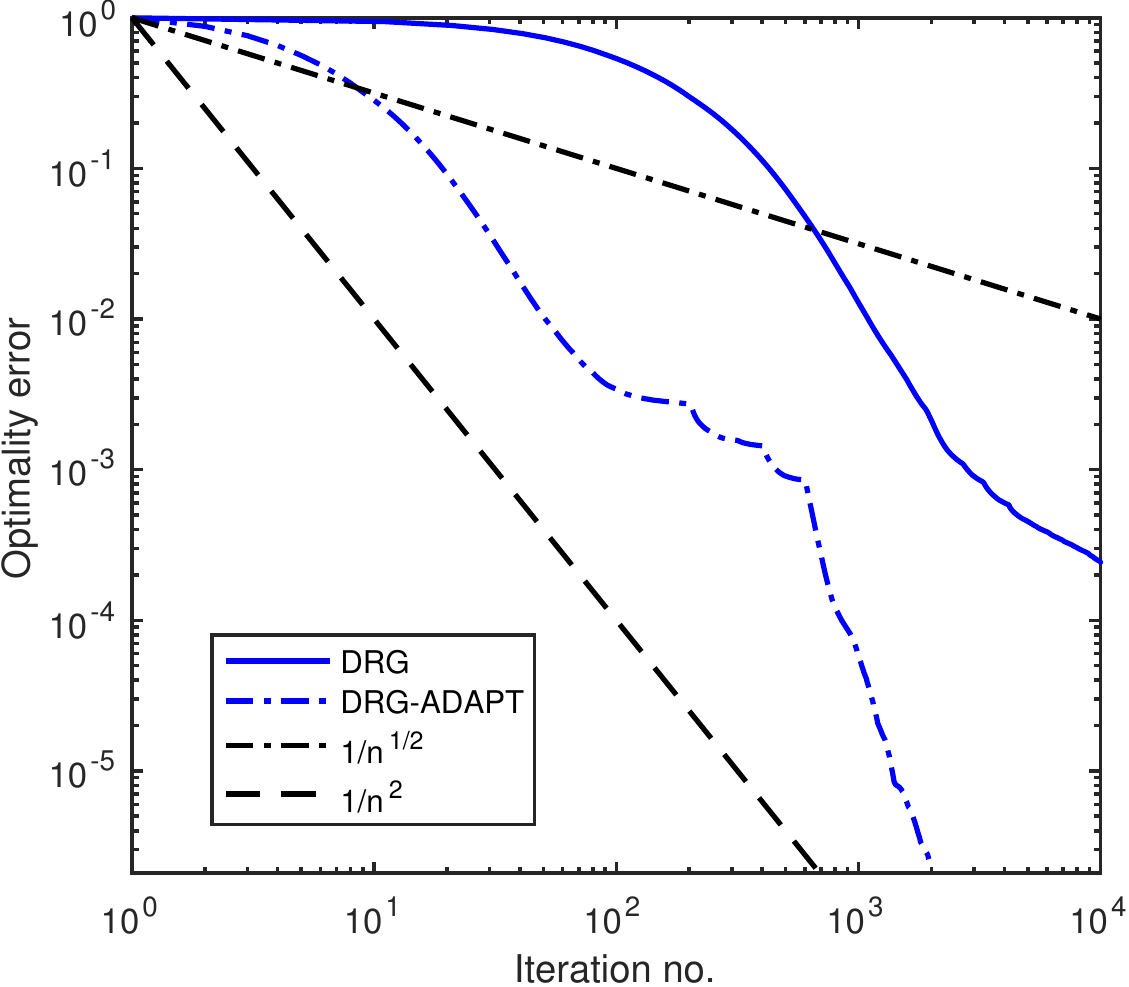}
        \caption{Logarithmic plot of optimality error $(V(u^k) - V^*)/(V(u^0) - V^*)$.}
        \label{fig:InSARConv}
\end{figure}
\cref{fig:InSAR} shows the result of applying TV denoising to an InSAR image of a slope of Mt. Vesuvius, Italy, with $\beta = 2$. The left column shows the phase data, while the right hand side shows the phase unwrapped data. 
The input image was taken from \cite{rocca1997overview}. 
It is evident that the algorithm is successful in removing noise. Computation time was 0.1 seconds per iteration on a 150$\times$150 image. A logarithmic plot showing convergence in terms of $(V(u^k) - V^*)/(V(u^0) - V^*)$ is shown in \cref{fig:InSARConv}, where $V^*$ is a near-optimal value for $V$, obtained by iterating until $V(u^{k+1}) - V(u^k) \leq 10^{-15}$. 
The plot shows the behaviour of \cref{alg:DG} with constant time steps $\tau_k = \tau_0 = 0.002$ and an ad-hoc adaptive method with $\tau_0 = 0.005$ where $\tau_k$ is halved each 200 iterations; for each of these strategies a separate $V^*$ was found since they did not produce convergence to the same minimizer. 
The reason for the different minimizers is that the TV functional, and thus the minimization problem, is non-convex in $S^1$ \cite{strekalovskiy2011total}. 
We can observe that the convergence speed varies between $\mathcal{O}(1/k)$ and $\mathcal{O}(1/k^2)$, with faster convergence for the ad-hoc adaptive method. 
The reason for this sublinear convergence as compared to the linear convergence observed in the Brockett flow case may be the non-convexity.

\subsubsection{DTI image denoising}
Diffusion Tensor Imaging (DTI) is a medical imaging technique where the goal is to make spatial samples of the tensor specifying the diffusion rates of water in biological tissue. The tensor is assumed to be, at each point $(i,j)$, represented by a matrix $A_{ij} \in \text{Sym}^+(3)$, the space of $3 \times 3$ symmetric positive definite (SPD) matrices. Experimental measurements of DTI data are, as with other MRI techniques, contaminated by Rician noise \cite{gudbjartsson1995rician}, which one may attempt to remove by minimizing \eqref{eq:TV_mfld} with an appropriate choice of Riemannian structure on $\mathcal{M} = \text{Sym}^+(3)^{m \times l}$.

As above, since the manifold we are working on is a product manifold, it suffices to define the Riemannian structure on $\text{Sym}^+(3)$. First off, one should note that $T_A\text{Sym}^+(3)$ can be identified with $\text{Sym}(3)$, the space of symmetric $3 \times 3$ matrices \cite{lang2012fundamentals}. In \cite{weinmann2014total}, the authors consider equipping $\text{Sym}^+(3)$ with the affine invariant Riemannian metric given pointwise as
\begin{align*}
g_{A}(X,Y) = \text{tr}(A^{-\frac{1}{2}}X A^{-1} Y A^{-\frac{1}{2}}),
\end{align*}
and for purposes of comparison, so shall we. The space $\text{Sym}^+(3)$ equipped with this metric is a Cartan-Hadamard manifold \cite{lang2012fundamentals}, and thus is complete, meaning that \cref{theo:convergence} holds. This metric induces the explicitly computable geodesic distance
\begin{align*}
d(A,B) = \sqrt{\sum \limits_{i = 1}^3 \log(\kappa_i)^2}
\end{align*}
on $\text{Sym}^+(3)$, where $\kappa_i$ are the eigenvalues of $A^{-\frac{1}{2}} B A^{-\frac{1}{2}}$. Furthermore, the metric induces a Riemannian exponential given by
\begin{align*}
\exp_A(Y) = A^{1/2} \mathrm{e}^{A^{-1/2} Y A^{-1/2}} A^{1/2}
\end{align*}
where $e$ denotes the matrix exponential, and $A^{1/2}$ is the matrix square root of $A$. We could choose the retraction as $\phi = \exp$, but there are less computationally expensive options that do not involve computing matrix exponentials. More specifically, we will make use of the second-order approximation of the exponential,
\begin{align*}
\phi_A(Y) = A + Y + \frac{1}{2} Y A^{-1} Y.
\end{align*}
While a first-order expansion is also a retraction, there is no guarantee that $A + Y \in \text{Sym}^+(3)$, whereas the second-order expansion, which can be written on the form
\begin{align*}
\phi_A(Y) = \frac{1}{2}A + \frac{1}{2}(A^{\frac{1}{2}}+A^{-\frac{1}{2}}Y)^T(A^{\frac{1}{2}}+A^{-\frac{1}{2}}Y),
\end{align*}
is clearly symmetric positive definite since $A$ is so. Note that using a sparse basis $E_{ij}$ (in our example we use $E_{ij} = e_i e_j^T + e_j e_i^T$) for the space $\text{Sym}(3)$, evaluating $\phi_A(X + \alpha E_{ij})$ amounts to, at most, four scalar updates when $\phi_A(X)$ and $A^{-1}$ is known, as is possible with proper bookkeeping in the software implementation. Also, since all matrices involved are $3\times3$ SPD matrices, one may find eigenvalues and eigenvectors directly, thus allowing for fast computations of matrix square roots and, consequently, geodesic distances. 
\begin{figure}[ht]
        \begin{minipage}{.43\textwidth}
                \centering
            \includegraphics[width=\linewidth]{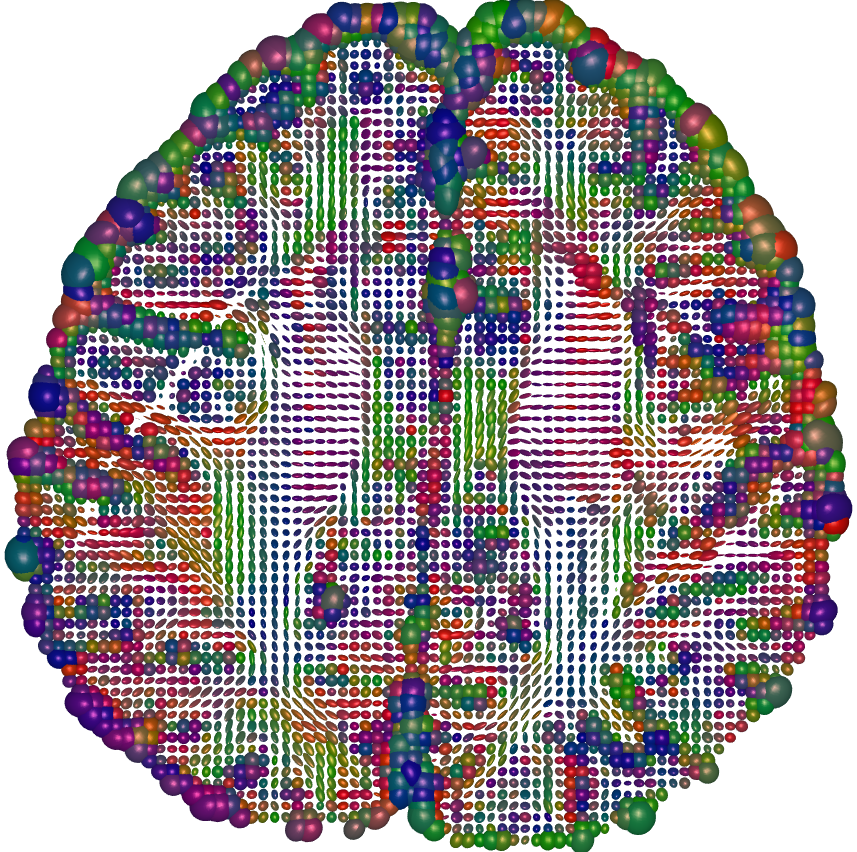}
        \end{minipage}
        \begin{minipage}{.43\textwidth}
        \includegraphics[width=\linewidth]{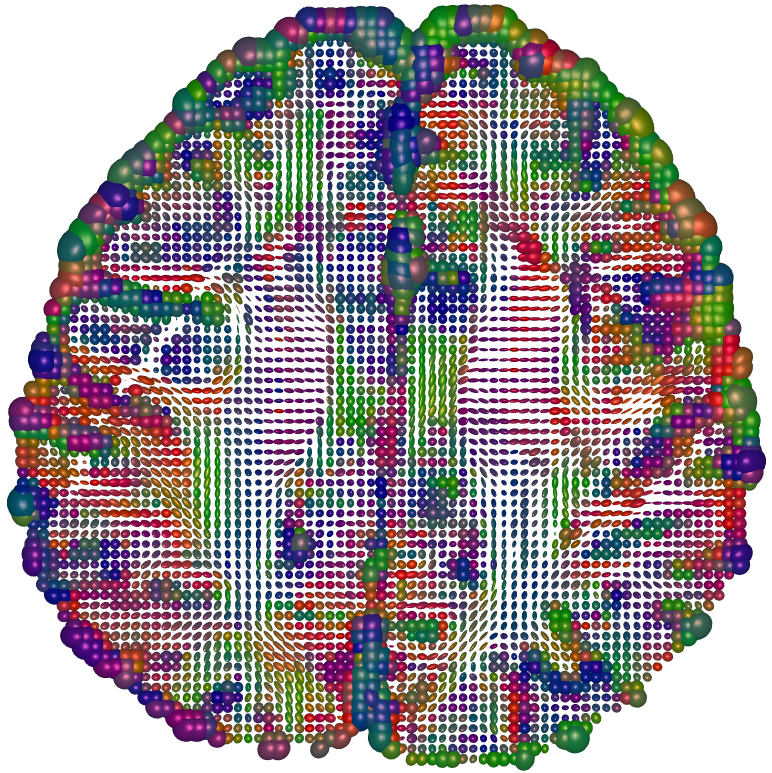}
        \end{minipage}
                \caption{DTI scan, axial slice. Left: Noisy image. Right: Denoised with $\beta = 2 $, $ \lambda = 0.05$.}
        \label{fig:DTI}
\end{figure}

\begin{figure}[t]
                \centering
        \begin{minipage}{.45\textwidth}
            \includegraphics[width=\linewidth]{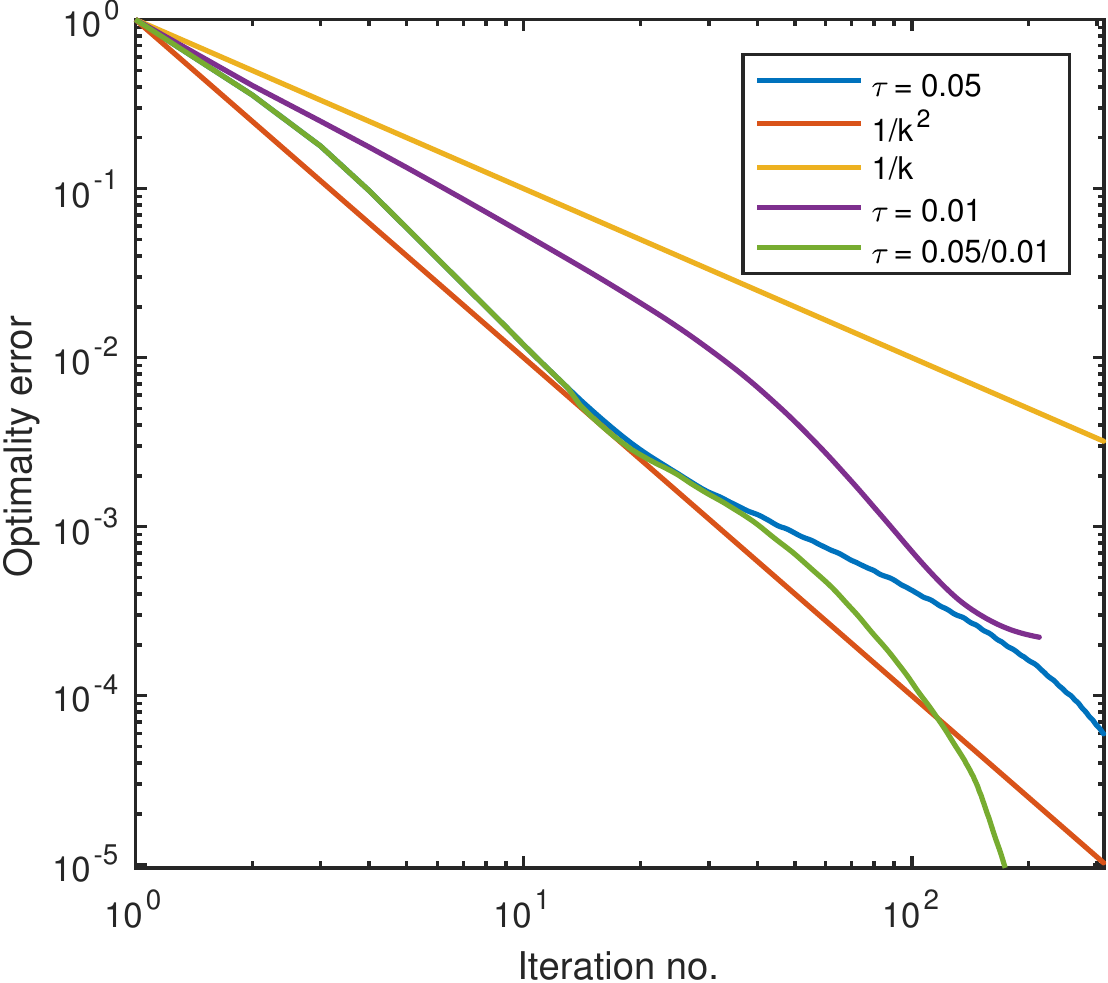}
        \end{minipage}
                \caption{Logarithmic plot of optimality error.}
        \label{fig:DTIConv}
\end{figure}

\cref{fig:DTI} shows an example of denoising DTI images using the TV regularizer. The data is taken from the publicly available Camino data set \cite{cook2006camino}. The DTI tensor has been calculated from underlying data using linear least-squares fitting, and is subject to Rician noise (left hand side), which is mitigated by TV denoising (right hand side). The denoising procedure took about 7 seconds for 57 iterations, on a 72$\times$73 image. The algorithm was stopped when the relative change in energy, $(V(u^0) - V(u^k))/V(u^0)$ dropped below $10^{-5}$. Each atom $A \in \text{Sym}^+(3)$ is visualized by an ellipsoid with the eigenvectors of $A$ as principal semi-axes, scaled by the corresponding eigenvalues. The colors are coded to correspond to the principal direction of the major axis, with red denoting left-right orientation, green anterior-posterior and blue inferior-superior. \cref{fig:DTIConv} shows the convergence behaviour of \cref{alg:DG}, with three different time steps: $\tau = 0.05$, $\tau = 0.01$ and a mixed strategy of using $\tau = 0.05$ for 12 steps, then changing to $\tau = 0.01$. Also, baseline rates of $1/k^2$ and $1/k$ are shown. It is apparent that the choice of time step has great impact on the convergence rate, and that simply changing the time step from $\tau = 0.05$ to $\tau = 0.01$ is effective in speeding up convergence. This would suggest that time step adaptivity is a promising route for acceleration of these methods.

\section{Conclusion and outlook}
We have extended discrete gradient methods to Riemannian manifolds, and shown how they may be applied to gradient flows. The Itoh--Abe discrete gradient has been formulated in a manifold setting; this is, to the best of our knowledge, the first time this has been done.  In particular, we have used the Itoh--Abe DRG on gradient systems to produce a derivative-free optimization algorithm on Riemannian manifolds. This optimization algorithm has been proven to converge under reasonable conditions, and shows promise when applied to the problem of denoising manifold valued images using the total variation approach of \cite{weinmann2014total}.

As with the algorithm in the Euclidian case, there are open questions. The first question is which convergence rate estimates can be made; one should especially consider the linear convergence exhibited in the Brockett flow problem, and the rate observed in \cref{fig:DTIConv} which approaches $1/k^2$. A second question is how to formulate a rule for choosing step sizes so as to accelerate convergence toward minimizers. There is also the question of how the DRG methods perform as ODE solvers for dissipative problems on Riemannian manifolds; in particular, convergence properties, stability, and convergence order. The above discussion is geared toward optimization applications due to the availability of optimization problems, but it would be of interest to see how the methods work as ODE solvers in their own right similar to the analysis and experiments done in \cite{celledoni2018energy}.

{\footnotesize
\bibliography{EP,tb}}
\bibliographystyle{siamplain}

\end{document}